\documentclass{amsproc}
\usepackage{amsmath}
\usepackage{amssymb}
\input xy
\xyoption{all}

\theoremstyle{definition}
\newtheorem{theorem}{Theorem}[section]

\theoremstyle{definition}

\newtheorem{example}[theorem]{Example}
\newtheorem{Lemma}[theorem]{Lemma}
\newtheorem{corollary}[theorem]{Corollary}
\newtheorem{Prop}[theorem]{Proposition}
\theoremstyle{remark}
\newtheorem{remark}[theorem]{Remark}

\numberwithin{equation}{section}
\newcommand{\Cal}[1]{{\mathcal #1}}
\newcommand{\codim}{\mbox{\rm codim}}
\newcommand{\End}{\operatorname{End}}
\newcommand{\Hom}{\operatorname{Hom}}
\DeclareMathOperator{\Ob}{Ob}
\newcommand{\Mod}{\operatorname{Mod-\!}}

\newcommand{\cmat}{\left(\begin{array}}
\newcommand{\fmat}{\end{array}\right)}

\newcommand{\Ui}{U^{(i)}}
\newcommand{\Vi}{V^{(i)}}
\newcommand{\U}[1]{U^{(#1)}}
\newcommand{\V}[1]{V^{(#1)}}

\begin{document}

 \title[On a category of chains of modules]{On a category of chains of modules whose endomorphism rings have at most $2n$ maximal ideals}
  
\author{Federico Campanini}
\address{Dipartimento di Matematica, Universit\`a di Padova, 35121 Padova, Italy}
\email{federico.campanini@math.unipd.it}

\thanks{Partially supported by Dipartimento di Matematica ``Tullio Levi-Civita'' of Universit\`a di Padova (Project BIRD163492/16 ``Categorical homological methods in the study of algebraic structures'' and Research program DOR1690814 ``Anelli e categorie di moduli'').}


\begin{abstract} We describe the endomorphism rings in an additive category whose objects are right $R$-modules $M$ with a fixed chain of submodules $0=M^{(0)}\leq M^{(1)}\leq M^{(2)} \leq \dots \leq M^{(n)}=M$ and the behaviour of these objects as far as their direct sums are concerned.
\end{abstract}

\maketitle

\section{Introduction}

Let $R$ be a ring, associative and with an identity 1, and  let $n$ be a fixed positive integer. In this paper, we focus our attention to a category $\Cal E_n$ whose objects are right $R$-modules $M$ with a fixed chain of submodules $0=M^{(0)}\leq M^{(1)}\leq M^{(2)} \leq \dots \leq M^{(n)}=M$. This category is a very natural generalization of the category $\Cal E$ studied in \cite{CF}. Indeed, the category $\Cal E$ in \cite{CF} is equivalent to the category $\Cal E_2$, and the techniques used in this article are essentially the same as those of \cite{CF}. Recall that a right $R$-module is \textit{uniserial} if the lattice of its submodules is linearly ordered under inclusion. The endomorphism ring $\End(U_R)$ of a non-zero uniserial module $U_R$ \cite[Theorem~1.2]{TAMS} has at most two maximal right ideals: the two-sided completely prime ideals  $I_U:=\{\,f\in \End(U_R) \mid f$ is not  injective$\,\}$ and
$K_U:=\{\,f\in \End(U_R) \mid f$ is not  surjective$\,\}$, or only one of them. Here, a {\em completely prime ideal} $P$ of a ring $S$ is a proper ideal $P$ of~$S$ such that, for every $x,y\in S$, $xy\in P$ implies that either $x\in P$ or $y\in P$. Two right $R$-modules $M_R$ and $N_R$ are said to have the same  {\em monogeny class}, denoted
by $[M_R]_m=[N_R]_m$, if there exist a monomorphism $M_R\rightarrow N_R$ and
a monomorphism $N_R\rightarrow M_R$. Similarly, $M_R$ and $N_R$ are said to have the same {\em epigeny class}, denoted
by $[M_R]_e=[N_R]_e$, if there exist an epimorphism $M_R\rightarrow N_R$ and
an epimorphism $N_R\rightarrow M_R$. For uniserial modules, the monogeny class and the epigeny class can be expressed in terms of the maximal right ideals of the endomorphism rings of the modules. To be more precise, two uniserial right $R$-modules $U$ and $V$ have the same monogeny class [resp. epigeny class] if and only if there exist two morphisms $f: U_R\rightarrow V_R$ and $g: V_R \rightarrow U_R$ such that $gf \notin I_U$ [resp. $gf \notin K_U$] (see Remark \ref{1.1}). Finite direct sums of uniserial modules are classified via their monogeny class and their epigeny class \cite{TAMS}. These two invariants are sufficient thanks to their relation with the maximal right ideals of the endomorphism ring of a uniserial module. Other classes of modules whose endomorphism rings have at most two maximal ideals have similar behaviours \cite{AAF1, Tufan, FG}.

In this paper, we generalize the concept of monogeny class and epigeny class, defining $2n$ classes for the objects of the category $\Cal E_n$ (Section \ref{classes}). We focus our attention to the objects $M$ of $\Cal E_n$ whose factors $M^{(i)}/M^{(i-1)}$ are uniserial right $R$-modules. For such objects, we find a behaviour very similar to that of uniserial modules mentioned in the previous paragraph. The main results of this paper are Theorem \ref{EARY}, in which we describe the endomorphism ring of these objects, and Theorem \ref{main}, in which finite direct sums of objects whose factors $M^{(i)}/M^{(i-1)}$ are uniserial right $R$-modules are classified via the $2n$ classes that generalize the monogeny class and the epigeny class.

In this article, rings are associative rings with identity and modules are right unitary modules. The term ``maximal ideal'' in a ring shall mean ``maximal two-sided ideal'', that is, ``maximal in the set of all proper two-sided ideals''.

\section{Objects with a semilocal endomorphism ring}

Let $R$ be a ring and let $n$ be a fixed positive integer. We consider the category $\Cal E_n$, defined as follows. The objects of $\Cal E_n$ are right $R$-modules $M$ with a fixed chain of submodules
$$
0=M^{(0)}\leq M^{(1)}\leq M^{(2)} \leq \dots \leq M^{(n)}=M.
$$
With abuse of notation, we simply denote by $M$ such an object. A morphism in $\Cal E_n$ between two chains $M$ and $N$ is a right $R$-module morphism $f:M\rightarrow N$ such that $f(M^{(i)})\subseteq N^{(i)}$ for every $i=1,\dots,n$.
We will denote by $E_M$ the endomorphism ring $\End_{\Cal E_n}(M)$ of an object $M$ in the category $\Cal E_n$.

Notice that $\Cal E_n$ is an additive category, whose zero object is given by the zero module with its trivial chain $0^{(i)}=0$ for every $i=1,\dots,n$. It will be denoted by~$0$.

Notice that a morphism $f:M\rightarrow N$ in $\Cal E_n$ induces right $R$-module morphisms on the factors
$$
f_i :\frac{M^{(i)}}{M^{(i-1)}} \longrightarrow \frac{N^{(i)}}{N^{(i-1)}}\quad \mbox{for every } i=1,\dots,n.
$$
We call $M^{(i)}/M^{(i-1)}$ the $i^{th}$\textit{-factor module} of $M$ and $f_i$ the $i^{th}$\textit{-induced morphism}.

We also define the full subcategory $\Cal U_n$ of the category $\Cal E_n$ consisting of the objects whose factor modules are all non-zero uniserial right $R$-modules. Examples of modules having a chain of submodules with all factor modules uniserial are modules of finite composition length, serial modules or more generally, the polyserial modules studied in \cite{Fu}, \cite{FS1} and \cite{FS2}.

\medskip

Since all the results in this paper derive essentially from the fact that the objects we are dealing with have a semilocal endomorphism ring or have an endomorphism ring of finite type, we briefly recall some notions and facts. A ring $R$ is {\em semilocal} if the factor ring $R/J(R)$, where $J(R)$ denotes the Jacobson radical of $R$, is a semisimple artinian ring. A ring $R$ is semilocal if and only if its dual Goldie dimension $\codim(R)$ is finite \cite[Proposition~2.43]{libro}. If $R$ and $S$ are rings, a ring morphism $\varphi \colon R \to S$
is  {\em local} if, for every $r\in R$, $\varphi (r)$ invertible in $S$ implies $r$ invertible in $R$~\cite{CD}.

A ring $R$ {\em has type $m$} \cite{AlbPav4}, if the factor ring $R/J(R)$ is a direct product of $m$ division rings, and that $R$ is a {\em ring of finite type} if it has type $m$ for some integer $m\ge 1$. If a ring $R$ has finite type, then the type $m$ of $R$ coincides with the dual Goldie dimension $\codim(R)$ of $R$ \cite[Proposition~2.43]{libro}. A module $M_R$ has {\em type $m$} ({\em finite type}) if its endomorphism ring $\End(M_R)$ is a ring of type $m$ (of finite type, respectively). In the particular case where $M_R$ is the zero module, its endomorphism ring is the trivial ring with one element. We define such a ring to be of dual Goldie dimension $0$ and we consider it to be a semilocal ring of finite type $0$.

\begin{Prop}\label{semilocal}
Let $M$ be an object of $\Cal E_n$, with factor modules $U^{(i)}=M^{(i)}/M^{(i-1)}$, $i=1,\dots,n$. Then:
\begin{enumerate}
	\item[{\rm (a)}]
	$\codim(E_M)\leq \sum_{i=1}^n \codim(\End(U^{(i)}))$. 
	
	\item[{\rm (b)}]
	If $\End(U^{(i)})$ is semilocal for every
	$i=1\dots,n$, then the endomorphism ring $E_M$ is
	also a semilocal ring.
	
	\item[{\rm (c)}]
	If, for every $i=1,\dots,n$, $U^{(i)}$ is of type
	$m_i$, then the endomorphism ring $E_M$ of $M$ is of
	type $\leq m_1+\dots+m_n$.
\end{enumerate}
In particular, for every object $M$ of $\Cal U_n$, $E_M$ is of finite type.
\end{Prop}

\begin{proof} There is a canonical ring morphism $\varphi\colon E_M\to\End(U^{(1)})\times\dots\times\End(U^{(n)})$ defined by $f\mapsto (f_1,\dots,f_n)$, which is a local morphism. Now (a) follows from \cite[Corollary 2]{CD}, and (b) follows from the fact that a ring is semilocal if and only if its dual Goldie dimension is finite \cite[Proposition~2.43]{libro}.

For (c), suppose that for every $i=1,\dots,n$, $U^{(i)}$ is of type $m_i$. By \cite[Proposition~2.1]{AlbPav4}, there are local morphisms  $\End(U^{(i)})\to D^{(i)}_1\times\dots\times D^{(i)}_{m_i}$ for suitable division rings $D^{(i)}_{j_i}$, $i=1,\dots,n$, $j_i=1,\dots, m_i$. Composing with $\varphi$, we get a local morphism from $E_M$ into $m_1+\dots+m_n$ division rings. Thus $E_M$ has type $\le m_1+\dots+m_n$ by \cite[Proposition~2.1]{AlbPav4}.

For the last assertion, recall that uniserial modules are of finite type.
\end{proof}

\begin{theorem}\label{EARY} Let $M$ be an object of $\Cal U_n$, with factor modules $U^{(i)}$.
For every $i=1,\dots,n$, set 
$$
I_{M,i,m}:=\{\,f\in E_M \mid f_i\mbox{ is not an injective right }R\mbox{-module morphism}\}
$$
and 
$$
I_{M,i,e}:=\{\,f\in E_M \mid f_i\mbox{ is not a surjective right }R\mbox{-module morphism}\}.
$$
Then $I_{M,1,m},\dots, I_{M,n,m},I_{M,1,e},\dots,I_{M,n,e}$ are $2n$ two-sided completely prime ideals of $E_M$, and
every proper right ideal of $E_M$ and every proper left ideal of $E_M$ is contained in one of these $2n$ ideals of $E_M$.
Moreover, $E_M/J(E_M)$ is isomorphic to the direct product of $k$ division rings $D_1,\dots,D_k$ for some $k$ with $1\leq k\le 2n$, where $\{D_1,\dots,D_k\}\subseteq \{E_M/I_{M,1,m},\dots,E_M/I_{M,n,m},E_M/I_{M,1,e},\dots,E_M/I_{M,n,e}\}$.
\end{theorem}

\begin{proof} As in the proof of Proposition~\ref{semilocal}, there is a canonical local morphism $\varphi\colon E_M\to\End(U^{(1)})\times\dots\times\End(U^{(n)})$. Since the modules $U^{(i)}$ are all uniserial modules, the canonical projections define local morphisms $\End(U^{(i)})\to \End(U^{(i)})/I_{U^{(i)}}\times \End(U^{(i)})/K_{U^{(i)}}$, where $\End(U^{(i)})/I_{U^{(i)}}$ and $\End(U^{(i)})/K_{U^{(i)}}$ are integral domains. More precisely, for every $i=1,\dots,n$ three cases can occur:

(1) $I_{U^{(i)}}\subseteq K_{U^{(i)}}$. In this case, $\End(U^{(i)})$ is a local ring with maximal ideal $K_{U^{(i)}}$, $\End(U^{(i)})/K_{U^{(i)}}$ is a division ring, and the canonical projection $$
\End(U^{(i)})\to\End(U^{(i)})/ K_{U^{(i)}}
$$
is a local morphism. 

(2) $K_{U^{(i)}}\subseteq I_{U^{(i)}}$. In this case, $\End(U^{(i)})$ is a local ring with maximal ideal $I_{U^{(i)}}$, $\End(U^{(i)})/I_{U^{(i)}}$ is a division ring, and the canonical projection
$$
\End(U^{(i)})\to \End(U^{(i)})/I_{U^{(i)}}
$$
is a local morphism. 

(3) $I_{U^{(i)}}\nsubseteq K_{U^{(i)}}$ and $K_{U^{(i)}}\nsubseteq I_{U^{(i)}}$. In this case, $\End(U^{(i)})$ has two maximal right ideals $I_{U^{(i)}}$ and $K_{U^{(i)}}$, which are two-sided ideals, $\End(U^{(i)})/I_{U^{(i)}}$ and $\End({U^{(i)}})/K_{U^{(i)}}$ are two division rings, and the canonical map
$$
\End(U^{(i)})\to \End(U^{(i)})/I_{U^{(i)}}\times \End(U^{(i)})/K_{U^{(i)}}
$$
is a local morphism. 

Thus, if $D_1,\dots,D_k$ are those $k$ rings between the rings $\End(U^{(i)})/I_{U^{(i)}}$ and $\End(U^{(i)})/K_{U^{(i)}}$ that are division rings, we get a local morphism $E_M\to D_1\times\dots\times D_k$ of $E_M$ into $k\leq 2n$ division rings. The $2n$ ideals $I_{M,i,m}$ and $I_{M,i,e}$ are the kernels of the canonical morphisms $E_M\to \End(U^{(i)})/I_{U^{(i)}}$ and $E_M\to \End(U^{(i)})/K_{U^{(i)}}$, respectively. It is therefore immediate that they are completely prime two-sided ideals of $E_M$.

For every $i=1,\dots,n$, the non-invertible elements of $\End(U^{(i)})$ are exactly the elements of $I_{U^{(i)}}\cup K_{U^{(i)}}$. Since $\varphi\colon E_M\to\End(U^{(1)})\times\dots\times\End(U^{(n)})$ is a local morphism, the non-invertible elements of $E_M$ are exactly the elements of $I_{M,1,m}\cup \dots \cup I_{M,n,m}\cup I_{M,1,e}\cup\dots\cup I_{M,n,e}$. But these $2n$ ideals of $E_M$ are completely prime, so that every proper right (or left) ideal of $E_M$, which necessarily consists of non-invertible elements, must be contained in one of them, by the Prime Avoidance Lemma. In particular, every maximal two-sided ideal of $E_M$ must be one of $I_{M,1,m},\dots, I_{M,n,m},I_{M,1,e},\dots,I_{M,n,e}$. Finally, the last assertion follows by applying the Chinese Remainder Theorem.
\end{proof}

\begin{remark} Notice that, in Theorem~\ref{EARY}, without the assumption $\Ui \neq 0$, the sets $I_{M,i,m}$ and $I_{M,i,e}$ defined in the statement could be empty. To be more precise, for a fixed $i=1,\dots,n$, the sets $I_{M,i,m}$ and $I_{M,i,e}$ are empty exactly when $\Ui=0$. However, under the hypothesis of Theorem~\ref{EARY}, the zero morphism belongs to $I_{M,1,m},\dots, I_{M,n,m},I_{M,1,e},\dots,I_{M,n,e}$ and these $2n$ sets are always proper ideals of $E_M$.
\end{remark}

\section{i-th monogeny class and i-th epigeny class}\label{classes}
Let $M$ and $N$ be two objects of $\Cal E_n$ with factor modules $\U{1},\dots,\U{n}$ and $\V{1},\dots, \V{n}$ respectively. For $i=1,\dots,n$, we will say that 
$M$ and $N$ \textit{belong to the same $i^{th}$ monogeny class}, and we will write $[M]_{i,m}=[N]_{i,m}$, if there exist two morphisms $f\colon M\to N$ and $g\colon N\to M$ in the category $\Cal E_n$ such that $f_i:\U{i}\rightarrow \V{i}$ and $g_i:\V{i}\rightarrow \U{i}$ are injective right $R$-module morphisms. Similarly, we will say that
$M$ and $N$ \textit{belong to the same $i^{th}$ epigeny class}, and write $[M]_{i,e}=[N]_{i,e}$ if there exist two morphisms $f\colon M\to N$ and $g\colon N\to M$ in the category $\Cal E_n$ such that $f_i:\U{i}\rightarrow \V{i}$ and $g_i:\V{i}\rightarrow \U{i}$ are surjective right $R$-module morphisms.

Notice that, in this notation, $[M]_{i,m}=[0]_{i,m}$ if and only if $[M]_{i,e}=[0]_{i,e}$, if and only if $\Ui$ is the zero module.

\begin{example}\label{monoepi}
Let $M$ and $N$ be two objects of $\Cal E_n$ with factor modules $\U{1},\dots,$ $\U{n}$ and $\V{1},\dots, \V{n}$ respectively. If $[M]_{i,m}=[N]_{i,m}$ for some $i=1,\dots,n$, then $[\Ui]_m=[\Vi]_m$, that is, $\Ui$ and $\Vi$ belong to the same monogeny class in the classical sense of \cite{TAMS}. Similarly, if $[M]_{i,e}=[N]_{i,e}$ for some $i=1,\dots,n$, then $[\Ui]_e=[\Vi]_e$. Now we will see an example in which the converse holds as well, that is, if $[\Ui]_m=[\Vi]_m$ for some $i=1,\dots,n$, then $[M]_{i,m}=[N]_{i,m}$, and if $[\Ui]_e=[\Vi]_e$ for some $i=1,\dots,n$, then $[M]_{i,e}=[N]_{i,e}$. Let $M_1,\dots,M_n$ and $N_1,\dots,N_n$ be right $R$-modules. Let $M$ and $N$ be the objects of $\Cal E_n$ defined as
$$
0< M_1 < M_1 \oplus M_2 < \dots < M_1 \oplus \dots \oplus M_n=M
$$
and
$$
0< N_1 <N_1 \oplus N_2 < \dots < \dots < N_1 \oplus \dots  \oplus N_n=N
$$
respectively. Then the $i^{th}$-factor module $\Ui$ of $M$ is isomorphic to $M_i$ and the $i^{th}$-factor module $\Vi$ of $N$ is isomorphic to $N_i$. So, if $[\Ui]_m=[\Vi]_m$ for some $i=1,\dots,n$, then there exist two injective right $R$-module morphisms $\varphi:M_i\rightarrow N_i$ and $\psi:N_i\rightarrow M_i$ that can be extended trivially on all the other direct summands of $M$ and $N$, in order to get two morphisms $f:M\rightarrow N$ and $g:N\rightarrow M$ in $\Cal E_n$ such that both the induced morphisms $f_i$ and $g_i$ are right $R$-module monomorphisms. Hence $[M]_{i,m}=[N]_{i,m}$. Similarly, if $[\Ui]_e=[\Vi]_e$, then $[M]_{i,e}=[N]_{i,e}$.
\end{example}
 
\begin{remark}\label{1.1} Let $M$ and $N$ be two objects of $\Cal U_n$ with factor modules $\U{1},\dots,$ $\U{n}$ and $\V{1},\dots, \V{n}$ respectively. For $a=m,e$, by \cite[Lemma~1.1]{TAMS}, $[M]_{i,a}=[N]_{i,a}$ if and only if there exist two morphisms $f\colon M\to N$ and $g\colon N \to M$ in the category $\Cal U_n$ such that $g f \notin I_{M,i,a}$ (or, equivalently, such that $f g \notin I_{N,i,a}$).\end{remark}
 
\begin{Lemma} \label{4.2} Let $M$ and $N$ be two objects of $\Cal U_n$. Fix $i=1,\dots,n$ and $a=m,e$ and assume that $[M]_{i,a}=[N]_{i,a}$. Then the following properties hold:
\begin{enumerate}
\item[{\rm (a)}]
For every $j=1,\dots,n$ and $b=m,e$, one has that $I_{M,j,b}\subseteq I_{M,i,a}$ if and only if $I_{N,j,b}\subseteq I_{N,i,a}$.

\item[{\rm (b)}]
For every $j=1,\dots,n$ and $b=m,e$, $I_{M,j,b}\subseteq I_{M,i,a}$ implies $[M]_{j,b}=[N]_{j,b}$.

\end{enumerate}
\end{Lemma}

\begin{proof}
{\rm (a)}
It suffices to show that if $j=1,\dots,n$, $b=m,e$ and $I_{M,j,b}\subseteq I_{M,i,a}$, then $I_{N,j,b}\subseteq I_{N,i,a}$. Fix $j=1,\dots,n$ and $b=m,e$ and suppose that $I_{N,j,b}\nsubseteq  I_{N,i,a}$. Let $\varphi\colon N\to N$ be a morphism in $I_{N,j,b}$ not in $I_{N,i,a}$. Since $[M]_{i,a}=[N]_{i,a}$, there exist two morphisms $f\colon M\to N$ and $g\colon N \to M$ in the category $\Cal U_n$ such that $g f \notin I_{M,i,a}$. In particular, by \cite[Lemma~1.1]{TAMS}, $g\varphi f\in I_{M,j,b}$ and $g\varphi f \notin  I_{M,i,a}$, which implies that $I_{M,j,b}\nsubseteq  I_{M,i,a}$.

{\rm (b)} Suppose that $I_{M,j,b}\subseteq I_{M,i,a}$ for some $j=1,\dots,n$ and $b=m,e$. Since $[M]_{i,a}=[N]_{i,a}$, there exist two morphisms $f\colon M\to N$ and $g\colon N \to M$ in the category $\Cal U_n$ such that $g f \notin I_{M,i,a}$. In particular $g f \notin I_{M,j,b}$, and therefore $[M]_{j,b}=[N]_{j,b}$ by Remark \ref{1.1}.

\end{proof}

\begin{corollary}\label{4.3}
Let $M$ and $N$ be two objects of $\Cal U_n$. Suppose that $[M]_{i,a}=[N]_{i,a}$ for every $i=1,\dots,n$ and every $a=m,e$. Consider the sets $\Cal S_M:=\{I_{M,i,a}\mid i=1,\dots,n\mbox{ and }a=m,e\}$ and $\Cal S_N:=\{I_{N,i,a}\mid i=1,\dots,n\mbox{ and }a=m,e\}$ partially ordered by set inclusion. Then the canonical mapping $\Phi\colon \Cal S_M\to \Cal S_N$, defined by 
$I_{M,i,a}\mapsto I_{N,i,a}$, is a partially ordered set isomorphism.
\end{corollary}

\begin{proof}
It immediately follows from Lemma \ref{4.2} {\rm (a)}.
\end{proof}

\section{Factor categories and finite direct sums}

Let $\Cal C$ be any preadditive category. An {\em ideal} $\Cal I$ of $\Cal C$ assigns to every pair $A,B$ of objects of $\Cal C$ a subgroup $\Cal I(A,B)$ of the abelian group $\Hom_{\Cal C}(A,B)$ with the property that, for all morphisms $\varphi\colon C\rightarrow A$, $\psi\colon A\rightarrow B$ and $\omega\colon B\rightarrow D$ with $\psi \in \Cal I(A,B)$, one has that $\omega\psi\varphi \in \Cal I(C,D)$ \cite[p.~18]{Several}.

Let $A$ be an object of $\Cal C$ and $I$ be a two-sided ideal of the ring $\End_{\Cal C}(A)$. Let $\Cal I$ be the ideal of the category $\Cal C$ defined as follows. A morphism $f \colon X \to Y$ in $\Cal C$ belongs to $\Cal I(X,Y)$ if $\beta f \alpha \in I$ for every pair of morphisms $\alpha \colon
A \to X$ and $\beta \colon Y \to A$ in~$\Cal C$. The ideal $\Cal I$ is called the {\em ideal of $\Cal C$ associated to}~$I$ \cite{AlbPav3, AlbPav4}. It is the greatest of the ideals $\Cal I'$ of $\Cal C$ with $\Cal I'(A,A)\subseteq I$. It is easily seen that $\Cal I(A,A) = I$. 
If $A$ is an object of $\Cal C$, the ideals associated to two distinct ideals of $\End_{\Cal C}(A)$ are obviously two distinct ideals of the category~$\Cal C$.

For any ideal $\Cal I$ of $\Cal E_n$, $F \colon \Cal E_n \to \Cal E_n / \Cal I$ will denote the canonical functor. Recall that the factor category $\Cal E_n/\Cal I$ has the same objects as $\Cal E_n$ and, for $M,N\in\Ob(\Cal E_n)=\Ob(\Cal E_n/\Cal I)$, the group of morphisms $M\to N$ in $\Cal E_n/\Cal I$ is defined to be the factor group $\Hom_{\Cal E_n}(M,N)/\Cal I(M,N)$.

For any object $M$ of $\Cal E_n$, if $I_{M,i,a}$ is a maximal right ideal of $E_M$, we denote by $\Cal I_{M,i,a}$ the ideal of $\Cal E_n$ associated to $I_{M,i,a}$. Set $V(M):=\{\, \Cal I_{M,i,a}\mid I_{M,i,a}$ is a maximal ideal of $E_M\,\}$, so that $V(M)$ has at most $2n$ elements (Theorem~\ref{EARY}).

\begin{Lemma}\label{star} Let $M$ be an object of $\Cal U_n$, $I_{M,i,a}$ be a maximal ideal of $E_M$, $\Cal I_{M,i,a}$ its associated ideal in $\Cal E_n$ and $F \colon \Cal E_n \to \Cal E_n/\Cal I_{M,i,a}$ be the canonical functor. 
\begin{enumerate}
\item[{\rm (a)}]
For any non-zero object $N$ of $\Cal E_n$ such that $E_N$ is a semilocal ring, either $\Cal I_{M,i,a}(N,N) = E_N$
or $\Cal I_{M,i,a}(N,N)$ is a maximal ideal of $E_N$. In this second case, the ideal of $\Cal E_n$ associated to $\Cal I_{M,i,a}(N,N)$ is equal to $\Cal I_{M,i,a}$.
\item[{\rm (b)}]
For any object $N$ of $\Cal E_n$ such that $E_N$ is of finite type, either $F(N) = 0$ or $F(N) \cong  F(M)$. Moreover, if $[M]_{i,a}=[N]_{i,a}$, then $F(N) \cong  F(M)$.
\item[{\rm (c)}]
If $F(M)^t\cong F(M)^s$ in the factor category $\Cal E_n/\Cal I_{M,i,a}$ for integers $t,s\ge 0$, then $t=s$.
\end{enumerate}
\end{Lemma}

\begin{proof} 
{\rm (a)} is an immediate consequence of \cite[Lemma~2.1(ii)]{FacPer1}.
 
{\rm (b)} From the previous part, we know that either $F(N) = 0$ or the endomorphism ring of $F(N)$ is a division ring. Let us consider the latter case. 
As $1_N\notin\Cal I_{M,i,a}(N,N)$, there are $\alpha \colon M \to N$ and $\beta \colon N \to M$ such that $\beta \alpha \notin I_{M,i,a}$. Thus $\beta(\alpha\beta)\alpha\notin I_{M,i,a}$. It follows that $\alpha\beta \notin \Cal I_{M,i,a}(N,N)$, 
Therefore $F(\beta)F(\alpha)$  and $F(\alpha)F(\beta)$ are automorphisms, so $F(M) \cong  F(N)$. For the last assertion, note that, also in the case $[M]_{i,a}=[N]_{i,a}$, there are morphisms $\alpha \colon M \to N$ and $\beta \colon N \to M$ such that $\beta \alpha \notin I_{M,i,a}$. So, as before, we can conclude that $F(M)\cong F(N)$.

{\rm (c)} Let $D$ be the endomorphism ring of the object $F(M)$ of $\Cal E_n/\Cal I_{M,i,a}$. Apply the functor $\Hom(F(M),-)\colon \Cal E_n/\Cal I_{M,i,a}\to\Mod D$.
\end{proof}

\begin{corollary}\label{4.6}
Let $M$ be an object of $\Cal U_n$, $I_{M,i,a}$ be a maximal ideal of $E_M$, $\Cal I_{M,i,a}$ its associated ideal in $\Cal E_n$ and $F \colon \Cal E_n \to \Cal E_n/\Cal I_{M,i,a}$ be the canonical functor. 
{\rm (1)} For any object $N$ of $\Cal U_n$, the following conditions are equivalent:

\begin{enumerate}
\item[{\rm (a)}]
$F(N)= 0$ in $\Cal E_n/\Cal I_{M,i,a}$.
\item[{\rm (b)}]
$\Cal I_{M,i,a}(N,N)=E_N$.
\item[{\rm (c)}]
$I_{N,i,a}\subsetneq \Cal I_{M,i,a}(N,N)$.
\item[{\rm (d)}]
$[M]_{i,a}\neq[N]_{i,a}$.
\end{enumerate}
{\rm (2)} For any object $N$ of $\Cal U_n$, the following conditions are equivalent:
\begin{enumerate}
\item[{\rm (a)}]
$F(M)\cong F(N)$ in $\Cal E_n/\Cal I_{M,i,a}$.
\item[{\rm (b)}]
$\Cal I_{M,i,a}(N,N)$ is a proper ideal of $E_N$.
\item[{\rm (c)}]
$I_{N,i,a}= \Cal I_{M,i,a}(N,N)$.
\item[{\rm (d)}]
$[M]_{i,a}=[N]_{i,a}$.
\end{enumerate}

Moreover, in this second case, $I_{N,i,a}$ is a maximal right ideal of $E_N$ and $\Cal I_{M,i,a}=\Cal I_{N,i,a}$.

\end{corollary}

\begin{proof} (1).
(a)${}\Leftrightarrow{}$(b) $F(N)= 0$ if and only if the endomorphism ring $E_N/\Cal I_{M,i,a}(N,N)$ of $F(N)$ in the factor category $\Cal E_n/ \Cal I_{M,i,a}$ is the zero ring.

(b)${}\Rightarrow{}$(c) It suffices to note that $I_{N,i,a}$ is always a proper ideal of $E_N$.

(c)${}\Rightarrow{}$(d) Let $\varphi$ be an element in $\Cal I_{M,i,a}(N,N)$ not in $I_{N,i,a}$. For any two morphisms $f\colon M\to N$ and $g\colon N\to M$ in the category $\Cal E_n$, one has $g\varphi f \in I_{M,i,a}$, by definition of associated ideal. Since $\varphi \notin I_{N,i,a}$, it follows from \cite[Lemma~1.1]{TAMS} that $gf \in I_{M,i,a}$. This means that $[M]_{i,a}\neq[N]_{i,a}$.

(d)${}\Rightarrow{}$(b)  By Remark \ref{1.1}, if $[M]_{i,a}\ne[N]_{i,a}$, then, for any morphism $f\colon M\to N$ and $g\colon N \to M$ in the category $\Cal U$, we have that $g1_N f=g f \in I_{M,i,a}$.
Therefore $1_N \in \Cal I_{M,i,a}(N,N)$, so that $\Cal I_{M,i,a}(N,N)=E_N$.

(2).
First notice that, by Lemma \ref{star} {\rm (b)}, either $F(M)=0$ or $F(M)\cong F(N)$ in $\Cal E_n/\Cal I_{M,i,a}$. Moreover, $I_{N,i,a}$ is always contained in $\Cal I_{M,i,a}(N,N)$. As a matter of fact, suppose $\varphi \in I_{N,i,a}$. By \cite[Lemma~1.1]{TAMS}, for any $f \colon M\to N$ and any $g \colon N\to M$, we have that $g\varphi f \in I_{M,i,a}$, so $\varphi \in I_{M,i,a}(N,N)$. Now all the implications follow from part (1).

For the last assertion, apply \cite[Lemma~2.1(ii)]{FacPer1}.
\end{proof}

\begin{corollary}\label{cor4.8}
Let $M$ and $N$ be two objects in the category $\Cal U_n$. Fix $i=1,\dots,n$ and $a=m,e$. Suppose that $[M]_{i,a}=[N]_{i,a}$. Then the following properties hold:

\item[{\rm (a)}]
$I_{M,i,a}$ is a maximal right ideal of $E_M$  if and only if $I_{N,i,a}$ is a maximal right ideal of $E_N$.

\item[{\rm (b)}]
Suppose that $I_{M,i,a}$ is a maximal right ideal of $E_M$. Then, for every $j=1,\dots,n$ and $b=m,e$, $I_{M,j,b}= I_{M,i,a}$ if and only if $I_{N,j,b}= I_{N,i,a}$.
\end{corollary}

\begin{proof}
{\rm (a)} It suffices to show that $I_{M,i,a}$ maximal implies $I_{N,i,a}$ maximal. From Corollary \ref{4.6} (2), $I_{N,i,a}=\Cal I_{M,i,a}(N,N)$ is a proper ideal of $E_N$, and it is a maximal right ideal of $E_N$ by Lemma \ref{star} {\rm (a)}.

{\rm (b)} By {\rm (a)}, the hypotheses on $M$ and $N$ are symmetrical. Therefore it suffices to show that $I_{M,j,b}= I_{M,i,a}$ implies $I_{N,j,b}= I_{N,i,a}$. The inclusion $I_{N,j,b}\subseteq I_{N,i,a}$ follows from Lemma \ref{4.2} {\rm (a)}. Moreover, by Lemma \ref{4.2} {\rm (b)}, we can interchange the role of $(i,a)$ and $(j,b)$ and deduce the opposite inclusion applying Lemma \ref{4.2} {\rm (a)} again.
\end{proof}

\begin{remark}\label{sumsemi}
Let $M$ and $N$ be two objects of $\Cal U_n$ with factor modules $\U{1}, \dots,$ $\U{n}$ and $\V{1},\dots,\V{n}$ respectively. The $i^{th}$-factor module of $M\oplus N$ is isomorphic to $\Ui \oplus \Vi$ and $\codim(\End(\Ui \oplus \Vi))=\codim(\End(\Ui))+\codim(\End(\Vi))$ \cite[Introduction, p. 575] {FacHer}. By Proposition \ref{semilocal} (a), $\codim(E_{M \oplus N})$ is finite, and therefore $E_{M\oplus N}$ is semilocal \cite[Proposition~2.43]{libro}.
\end{remark}

\begin{Lemma}\label{twosided}
Let $M_1,\dots,M_r$ be objects of $\Cal U_n$. Then every maximal two-sided ideal of $E_{\bigoplus_{k=1}^r M_k}$ is of the form
$\Cal I_{M_h,i,a}(\bigoplus_{k=1}^r M_k,\bigoplus_{k=1}^r M_k)$
for some $h=1,\dots,r$, $i=1,\dots,n$ and $a=m,e$ such that $I_{M_h,i,a}$ is a maximal right ideal of $E_{M_h}$. Conversely, if $(h,i,a)$ is a triple such that $h=1,\dots,r$, $i=1,\dots,n$, $a=m,e$ and $I_{M_h,i,a}$ is a maximal right ideal of $E_{M_h}$, then $\Cal I_{M_h,i,a}(\bigoplus_{k=1}^r M_k,\bigoplus_{k=1}^r M_k)$ is a maximal two-sided ideal of $E_{\bigoplus_{k=1}^r M_k}$.
\end{Lemma}

\begin{proof}
First, let $I$ be a maximal two-sided ideal of $E_{\oplus_{k=1}^r M_k}$ and let $\Cal I$ be its associated ideal on $\Cal E_n$. Using \cite[Lemma~2.1(ii)]{FacPer1}, we get that for any $M=M_1,\dots,M_t$, either $\Cal I(M,M)=E_M$ or $\Cal I(M,M)$ is a maximal two-sided ideal of $E_M$. If $\Cal I(M_h,M_h)=E_{M_h}$ for all $h=1,\dots,t$, by definition of associated ideal, it follows that $\varepsilon_h \pi_h \in I$ for every $h=1,\dots,t$, where $\varepsilon_h : M_h\rightarrow \oplus_{k=1}^r M_k$ and $\pi_h: \oplus_{k=1}^r M_k \rightarrow M_h$ are the canonical embedding and the canonical projection respectively. In particular $1_{\oplus_{k=1}^r M_k}=\sum_{h=1}^r \varepsilon_h \pi_h \in I$, which is absurd. It follows that there exists a triple $(h,i,a)$ such that $\Cal I(M_h,M_h)=I_{M_h,i,a}$ is a maximal right ideal of $E_{M_h}$. By \cite[Lemma~2.1(i)]{FacPer1}, $\Cal I_{M_h,i,a}=\Cal I$, so that, in particular, $I=\Cal I(\bigoplus_{k=1}^r M_k,\bigoplus_{k=1}^r M_k)=\Cal I_{M_h,i,a}(\bigoplus_{k=1}^r M_k,\bigoplus_{k=1}^r M_k)$.

Conversely, let $(h,i,a)$ be a triple such that $h=1,\dots,t$, $i=1,\dots,n$, $a=m,e$ and $I_{M_h,i,a}$ is a maximal right ideal of $E_{M_h}$. Denote by $\Cal Q$ the quotient category $\Cal E_n/\Cal I_{M_h,i,a}$ and consider the canonical functor $F:\Cal E_n\rightarrow \Cal Q$. Then $F(\bigoplus_{k=1}^r M_k)\cong F(M_h)^m$, where $m:=m_{h,i,a}=|\{k\mid k=1,\dots,r,\ [M_k]_{i,a}=[M_h]_{i,a}\}|$ (Corollary \ref{4.6}).
It follows that $\End_{\Cal Q}(\bigoplus_{k=1}^r M_k)\cong M_{m}(D)$, where $D=\End_{\Cal Q}(M_h)$ is a division ring. In particular, $\End_{\Cal Q}(\bigoplus_{k=1}^r M_k)$ is a simple artinian ring and so the kernel of $E_{\bigoplus_{k=1}^r M_k}\rightarrow \End_{\Cal Q}(\bigoplus_{k=1}^r M_k)$, which is $\Cal I_{M_h,i,a}(\bigoplus_{k=1}^r M_k,\bigoplus_{k=1}^r M_k)$, is a maximal two-sided ideal of $E_{\bigoplus_{k=1}^r M_k}$.

\end{proof}

\begin{corollary} Let $M_1,\dots,M_r$ be objects of $\Cal U_n$. Then there is a one-to-one correspondence between $\bigcup_{k=1}^r V(M_k)$ and the maximal two-sided ideals of $\bigoplus_{k=1}^r M_k$ given by
$$
\Psi: \Cal I_{M_k,i,a}\mapsto \Cal I_{M_k,i,a}(\oplus_{k=1}^r M_k,\oplus_{k=1}^r M_k).
$$
\end{corollary}

\begin{proof}
By Lemma \ref{twosided}, $\Psi$ is well defined. We can construct the inverse map of $\Psi$ as follows. Let $I$ be a maximal two-sided ideal of $E_{\oplus_{k=1}^r M_k}$ and let $\Cal I$ be its associated ideal on $\Cal E_n$. As in the proof of Lemma \ref{twosided}, there exists a triple $(h,i,a)$ such that $\Cal I(M_h,M_h)=I_{M_h,i,a}$ is a maximal right ideal of $E_{M_h}$ and $\Cal I_{M_h,i,a}=\Cal I$. So, $\Cal I \in \bigcup_{k=1}^r V(M_k)$ and $I\mapsto \Cal I$ is the inverse of $\Psi$ (apply \cite[Lemma~2.1(i)]{FacPer1}).
\end{proof}

In the proof of the next result, we will make use of some techniques, notations and ideas taken from \cite{DF}.

\begin{Prop}\label{parziale}
Let $M_1,M_2\dots,M_r,N_1,N_2,\dots,N_s$ be $r+s$ objects of $\Cal U_n$. Then $\bigoplus_{k=1}^r M_k \cong \bigoplus_{l=1}^s N_l$ in the category $\Cal E_n$ if and only if $r=s$ and there exist $2n$ permutations $\varphi_{i,a}$ of $\{1,2,\dots,r\}$, where $i=1,\dots,n$ and $a=m,e$, such that $[M_k]_{i,a}=[N_{\varphi_{i,a}(k)}]_{i,a}$ for every $k=1,\dots,r$.
\end{Prop}

\begin{proof}
First assume that $\bigoplus_{k=1}^r M_k\cong \bigoplus_{l=1}^s N_l$ in the category $\Cal E_n$. For every $i=1,\dots, n$, $k=1,\dots,r$ and $l=1,\dots, s$, let $\Ui_k$ and $\Vi_l$ denote the $i^{th}$-factor module of $M_k$ and $N_l$ respectively. Since $\bigoplus_{k=1}^r U_k^{(1)}\cong \bigoplus_{l=1}^s V_l^{(1)}$ in the category Mod-$R$, looking at the Goldie dimension, we get that $r=s$.
Let $\alpha: \bigoplus_{k=1}^r M_k \rightarrow \bigoplus_{l=1}^s N_l$ be an isomorphism in $\Cal E_n$ with inverse $\beta: \bigoplus_{l=1}^s N_l \rightarrow \bigoplus_{k=1}^r M_k$. Denote by $\varepsilon_h:M_h\rightarrow \bigoplus_{k=1}^r M_k$, $\pi_h: \bigoplus_{k=1}^r M_k \rightarrow M_h$, $\varepsilon'_j:N_j\rightarrow \bigoplus_{l=1}^s N_l$ and $\pi'_j: \bigoplus_{l=1}^s N_l \rightarrow N_j$ the embeddings and the canonical projections and consider the composite morphisms $\chi_{h,j}:=\pi'_j\alpha \varepsilon_h: M_h\rightarrow N_j$ and $\chi'_{j,h}:=\pi_h\beta\varepsilon'_j:N_j\rightarrow M_h$. Fix $i=1,\dots,n$. We want to prove the existence of the permutation $\varphi_{i,e}$ (dualizing the proof we get the existence of the permutation $\varphi_{i,m}$).
Define a bipartite digraph $D=D(X,Y;E)$ (according on the choice of the pair $(i,e)$) having $X=\{M_1,\dots,M_r\}$ and $Y=\{N_1,\dots,N_r\}$ as disjoint sets of non-adjacent vertices, and the set $E$ of edges defined as follows: one edge from $M_h$ to $N_j$ for each $h$ and $j$ such that the $i^{th}$-induced morphism $\chi^{(i)}_{h,j}: \Ui_h\rightarrow \Vi_j$ is a surjective right $R$-module morphism, and one edge from $N_j$ to $M_h$ for each $h$ and $j$ such that the $i^{th}$-induced morphism $\chi'^{(i)}_{j,h}: \Vi_j\rightarrow \Ui_h$ is a surjective right $R$-module morphism.

We want to show that for every subset $T\subseteq X\cup Y$ of vertices, $|T|\leq |N^+(T)|$, where $N^+(T) :=\{\,w\in V\mid (v,w)\in E$ 
for some $v\in T\,\}$ is the {\em out-neighbourhood} of $T$ (\cite[Introduction, p.184]{DF}). Since the digraph is bipartite, we can suppose that $T\subseteq X$. If $p=|T|$ and $q=|N^+(T)|$, relabeling the indices we may assume that $T=\{M_1,\dots, M_p\}$ and $N^+(T)=\{N_1,\dots,N_q\}$. It means that the induced morphisms $\chi^{(i)}_{h,j}$ are not surjective for every $h=1,\dots, p$  and every $j=q+1,\dots,r$. Since the modules $\Vi_j$ are all uniserial, we have that $L_j:=\bigcup_{h=1}^p \chi^{(i)}_{h,j}(\Ui_h)\subsetneq \Vi_j$ for every $j=q+1,\dots,n$, and therefore all the quotient modules $\Vi_j/L_j$ are non-zero for every $j=q+1,\dots,n$. Let $\pi: \bigoplus_{l=1}^r \Vi_l \rightarrow \bigoplus_{j=q+1}^r \Vi_j/L_j$ be the canonical projection. For every $h=1,\dots, p$ and for every $j=q+1,\dots,r$, the composite morphism
$$
\Ui_h\overset{\varepsilon^{(i)}_h}{\longrightarrow} \bigoplus_{k=1}^r \Ui_k
\overset{\alpha^{(i)}}{\longrightarrow} \bigoplus_{l=1}^r \Vi_l
\overset{\pi'^{(i)}_j}{\longrightarrow} \Vi_j \rightarrow \Vi_j/L_j
$$
is zero because $\pi'^{(i)}_j \alpha^{(i)} \varepsilon^{(i)}_h(\Ui_h)=\chi^{(i)}_{h,j}(\Ui_h)\subseteq L_j$. It follows that for every $h=1,\dots,p$, $\varepsilon^{(i)}_h(\Ui_h)$ is contained in the kernel of
$$
\bigoplus_{k=1}^r \Ui_k
\overset{\alpha^{(i)}}{\longrightarrow} \bigoplus_{l=1}^r \Vi_l
\overset{\pi'^{(i)}_j}{\longrightarrow} \Vi_j \rightarrow \Vi_j/L_j
$$
for every $j=q+1,\dots,r$. Since $\sum_{h=1}^p \varepsilon^{(i)}_h(\Ui_h)=\bigoplus_{h=1}^p \Ui_h$, it follows that there exists a morphism $\bigoplus_{k=1}^r \Ui_k/\bigoplus_{h=1}^p \Ui_h\cong \bigoplus_{m=p+1}^r \Ui_m \rightarrow \Vi_j/L_j$ making the following diagram
$$
\xymatrix{
\bigoplus_{k=1}^r \Ui_k \ar[r] \ar[d]_{\alpha^{(i)}}
 &  \bigoplus_{m=p+1}^r \Ui_m \ar[d] \\
\bigoplus_{l=1}^r \Vi_l \ar[r] & \Vi_j/L_j
}
$$
commute. Hence, there exist a morphism $\gamma: \bigoplus_{m=p+1}^r \Ui_m \rightarrow \bigoplus_{j=q+1}^r \Vi_j/L_j$ and a commutative diagram
$$
\xymatrix{
\bigoplus_{k=1}^r \Ui_k \ar[r] \ar[d]_{\alpha^{(i)}}
 &  \bigoplus_{m=p+1}^r \Ui_m \ar[d]^\gamma \\
\bigoplus_{l=1}^r \Vi_l \ar[r] & \bigoplus_{j=q+1}^r\Vi_j/L_j
}
$$
Since the horizontal arrows are the canonical projections, the morphism $\gamma$ must be surjective. Taking the dual Goldie dimension of the domain and the codomain of $\gamma$, we get $r-|T|\geq r-|N^+(T)|$, that is $|T|\leq |N^+(T)|$. To conclude, it suffices to apply \cite[Lemma~2.1]{DF} to the digraph $D$.

Conversely, assume that there exist $2n$ permutations $\varphi_{i,a}$ of $\{1,2,\dots,r\}$, where $i=1,\dots,n$ and $a=m,e$, such that $[M_k]_{i,a}=[N_{\varphi_{i,a}(k)}]_{i,a}$. Fix $M \in \{M_1,\dots,M_r$, $N_1,\dots, N_r\}$, $i=1,\dots,n$ and $a=m,e$ such that $I_{M,i,a}$ is a maximal right ideal of $E_M$. Consider the canonical functor $F:\Cal E_n\rightarrow \Cal E_n/\Cal I_{M,i,a}$. Then $F(\bigoplus_{k=1}^r M_k)\cong F(M)^m \cong F(\bigoplus_{l=1}^r N_l)$, where $m:=m_{M,i,a}=|\{k\mid k=1,\dots,r,\ [M_k]_{i,a}=[M]_{i,a}\}|=|\{l\mid l=1,\dots,r,\ [N_l]_{i,a}=[M]_{i,a}\}|$ (Corollary \ref{4.6}). It means that there exists $f_{M,i,a}:\bigoplus_{k=1}^r M_k \rightarrow \bigoplus_{l=1}^r N_l$ in $\Cal E_n$ which becomes an isomorphism in the factor category $\Cal E_n/\Cal I_{M,i,a}$.
Notice that the triple $(M,i,a)$ identifies a pair of objects $(M_k,N_{\varphi_{i,a}(k)})$ with the following properties: $[M_k]_{i,a}=[N_{\varphi_{i,a}(k)}]_{i,a}$, $I_{M_k,i,a}$ is a maximal right ideal of $E_{M_k}$, $I_{N_{\varphi_{i,a}(k)},i,a}$ is a maximal right ideal of $E_{N_{\varphi_{i,a}(k)}}$ and $\Cal I_{M_k,i,a}=\Cal I_{N_{\varphi_{i,a}(k)},i,a}$ (Corollary \ref{4.6} and Corollary \ref{cor4.8}). So, according to Lemma \ref{twosided}, the triple $(M,i,a)$ defines a maximal two-sided ideal $J_{M,i,a}$ of $E_{\bigoplus_{k=1}^r M_k}$, namely $J_{M,i,a}=\Cal I_{M,i,a}(\bigoplus_{k=1}^r M_k,\bigoplus_{k=1}^r M_k)$. Since $E_{\bigoplus_{k=1}^r M_k}$ is a semilocal ring, there exists $\delta_{M,i,a}$ such that $\delta_{M,i,a}\equiv 1$ modulo $J_{M,i,a}$ and $\delta_{M,i,a}\equiv 0$ modulo all the other maximal two-sided ideals of $E_{\bigoplus_{k=1}^r M_k}$. Similarly, we can define an endomorphism $\delta'_{M,i,a}$ of $\bigoplus_{l=1}^r N_l$.
Consider any subset $\Omega$ of $\{M_1,\dots,M_r,N_1,\dots,N_r\}\times\{1,\dots,n\}\times\{m,e\}$ consisting of all the triples $(M,i,a)$ such that $I_{M,i,a}$ is a maximal right ideal of $E_M$ with the additional property that if $(M,i,a)\neq(M',j,b)$, then $\Cal I_{M,i,a}\neq \Cal I_{M',j,b}$. So, in particular $\{\Cal I_{M,i,a}\mid (M,i,a)\in \Omega\}=\left(\bigcup_{k=1}^r V(M_k)\right)\cup \left(\bigcup_{l=1}^s V(N_l)\right)$. Define the morphism $f:=\sum_{(M,i,a)\in \Omega} \delta'_{M,i,a}f_{M,i,a}\delta_{M,i,a}$. By construction, $f$ becomes an isomorphism in the factor category $\Cal E_n/\Cal I_{M,i,a}$ for every $(M,i,a)\in \Omega$. By \cite[Proposition 5.2]{Adel}, $f$ is an isomorphism in the category $\Cal E_n/\Cal I$, where $\Cal I$ is the intersection of the finitely many ideals $\Cal I_{M,i,a}$, with $(M,i,a)\in \Omega$. Since $\bigoplus_{k=1}^r M_k$ and $\bigoplus_{l=1}^r N_l$ are isomorphic in the category $\Cal E_n/\Cal I$, there exist two morphisms $\alpha: \bigoplus_{k=1}^r M_k \rightarrow \bigoplus_{l=1}^r N_l$ and $\beta: \bigoplus_{l=1}^r N_l \rightarrow \bigoplus_{k=1}^r M_k$ such that $\beta \alpha \equiv 1_{\bigoplus_{k=1}^r M_k}$ modulo $\Cal I(\bigoplus_{k=1}^r M_k,\bigoplus_{k=1}^r M_k)$ and $\alpha\beta \equiv 1_{\bigoplus_{l=1}^r N_l}$ modulo $\Cal I(\bigoplus_{l=1}^r N_l,\bigoplus_{l=1}^r N_l)$. But $E_{\bigoplus_{k=1}^r M_k}$ is semilocal (Remark \ref{sumsemi}), so its Jacobson radical $J(E_{\bigoplus_{k=1}^r M_k})$ is equal to the intersection of all its maximal two-sided ideals, which are the ideals $\Cal I_{M,i,a}(\bigoplus_{k=1}^r M_k,\bigoplus_{k=1}^r M_k)$, where $(M,i,a) \in \Omega$. It means that
$$
\Cal I(\bigoplus_{k=1}^r M_k,\bigoplus_{k=1}^r M_k)=\bigcap_{(M,i,a)\in \Omega} \Cal I_{M,i,a}(\bigoplus_{k=1}^r M_k,\bigoplus_{k=1}^r M_k)=J(E_{\bigoplus_{k=1}^r M_k}).
$$
Similarly,
$$
\Cal I(\bigoplus_{l=1}^r N_l,\bigoplus_{l=1}^r N_l)=\bigcap_{(M,i,a)\in \Omega} \Cal I_{M,i,a}(\bigoplus_{l=1}^r N_l,\bigoplus_{l=1}^r N_l)=J(E_{\bigoplus_{l=1}^r N_l}).
$$
Thus $\alpha\beta$ and $\beta\alpha$ are invertible in the rings $E_{\bigoplus_{k=1}^r M_k}$ and $E_{\bigoplus_{l=1}^r N_l}$, respectively. In particular, $\alpha$ is both right invertible and left invertible in the category $\Cal E_n$. It follows that $\alpha$ is an isomorphism in $\Cal E_n$.

\end{proof}

\begin{corollary}
Let $M_1,M_2\dots,M_r,N_1,N_2,\dots,N_s$ be $r+s$ objects of $\Cal U_n$. For every $i=1,\dots,n$ and every $a=m,e$, define
$$
X_{i,a}:=\{k\mid k=1,\dots,r,\ I_{M_k,i,a} \mbox{ is a maximal ideal of } E_{M_k}\}
$$
and
$$
Y_{i,a}:=\{l\mid l=1,\dots,s,\ I_{N_l,i,a} \mbox{ is a maximal ideal of } E_{N_l}\}.
$$

Then $\bigoplus_{k=1}^r M_k \cong \bigoplus_{l=1}^s N_l$ in the category $\Cal E_n$ if and only if $r=s$ and there exist $2n$ bijections $\psi_{i,a}:X_{i,a}\rightarrow Y_{i,a}$, where $i=1,\dots,n$ and $a=m,e$, such that $[M_k]_{i,a}=[N_{\psi_{i,a}(k)}]_{i,a}$ for every $k\in X_{i,a}$.
\end{corollary}

\begin{proof}
If $\bigoplus_{k=1}^r M_k\cong \bigoplus_{l=1}^s N_l$ in the category $\Cal E_n$, then we can restrict the permutations $\varphi_{i,a}$ of Proposition \ref{parziale} to $X_{i,a}$, noting that $\varphi_{i,a}(X_{i,a})=Y_{i,a}$ (Corollary \ref{cor4.8}).

For the other implication, observe that the proof of the ``if part'' of Proposition \ref{parziale} works also in this case.
\end{proof}

\begin{corollary}
Let $M$ and $N$ be two objects of $\Cal U_n$. Then $M$ and $N$ are isomorphic in $\Cal U_n$ if and only if $[M]_{i,a}=[N]_{i,a}$ for all $i=1,\dots,n$ and $a=m,e$.
\end{corollary}

\begin{example}
Let $r\geq 2$ be an integer. In \cite[Example 2.1]{TAMS}, Alberto Facchini constructed an example of $r^2$ pairwise 
non-isomorphic finitely presented
uniserial modules $U_{j,k}$ ($j,k=1,2,\dots,r$) over a suitable serial ring $R$ in order to show that a module that is a direct sum of $r$ uniserial modules can have $r!$ pairwise non-isomorphic direct-sum decompositions into indecomposables. The modules $U_{j,k}$
satisfy the following properties:
\begin{enumerate}
\item[{\rm (a)}]
for every $j,k,h,l=1,2,\dots,r$,
$[U_{j,k}]_m=[U_{h,l}]_m$ if and only if $j=h$;
\item[{\rm (b)}]
for every $j,k,h,l=1,2,\dots,r$,
$[U_{j,k}]_e=[U_{h,l}]_e$ if and only if $k=l$.\end{enumerate}
(We refer to \cite{TAMS} for the construction of the modules $U_{j,k}$.) Following the spirit of this example and using the modules $U_{j,k}$, we want to show that the permutations $\varphi_{i,a}$ in the statement of Proposition \ref{parziale} can be completely arbitrary. For any choice of $2r$ elements $k_1,\dots,k_{2r}$ of $\{1,\dots,r\}$, define the following object $M_{k_1,\dots,k_{2r}}$ of $\Cal U_n$:
$$
0<U_{k_1,k_2}<U_{k_1,k_2}\oplus U_{k_3,k_4}< \dots <
U_{k_1,k_2}\oplus \dots \oplus U_{k_{2r-1},k_{2r}}=M_{k_1,\dots,k_{2r}}.
$$
For any two objects $M_{k_1,\dots,k_{2r}}$ and $M_{h_1,\dots,h_{2r}}$ of this form, using Example \ref{monoepi} and the properties {\rm (a)} and {\rm (b)}, we get
$
[M_{k_1,\dots,k_{2r}}]_{i,m}=[M_{h_1,\dots,h_{2r}}]_{i,m}
$
if and only if $[U_{k_{2i-1},k_{2i}}]_m=[U_{h_{2i-1},h_{2i}}]_m
$,
if and only if $k_{2i-1}=h_{2i-1}$ and
$
[M_{k_1,\dots,k_{2r}}]_{i,e}=[M_{h_1,\dots,h_{2r}}]_{i,e}
$
if and only if
$[U_{k_{2i-1},k_{2i}}]_m=[U_{h_{2i-1},h_{2i}}]_m$,
if and only if $k_{2i}=h_{2i}$. In particular, from Proposition \ref{parziale}, the $r!$ objects $M_{k_1,\dots,k_{2r}}$ of $\Cal U_n$ are pairwise non-isomorphic. Moreover, given $2n$ permutations $\sigma_1,\dots,\sigma_n,\tau_1\dots,\tau_n$ of $\{1,2,\dots,r\}$ we have the following isomorphism
$$
M_{\sigma_1(1),\tau_1(1),\dots,\sigma_r(1),\tau_r(1)}\oplus M_{\sigma_1(2),\tau_1(2),\dots \sigma_r(2),\tau_r(2)}\oplus \dots \oplus M_{\sigma_1(r),\tau_1(r),\dots,\sigma_r(r),\tau_r(r)} \cong
$$
$$
\cong M_{1,1,\dots,1}\oplus M_{2,2,\dots 2}\oplus \dots \oplus M_{r,r,\dots, r}
$$
and the bijections $\varphi_{i,m}$ and $\varphi_{j,e}$ in the statement of Proposition \ref{parziale} are the permutations $\sigma_i$ and $\tau_j$ respectively.

\end{example}

\section{General case}
In this section, we fix a simple right $R$-module $S$ and we denote by $S^n$ the following object of $\Cal E_n$:
$$
0<S<S\oplus S< S\oplus S\oplus S <\dots<
\underbrace{S\oplus\dots\oplus S}_{n\text{-times}}=S^n.
$$
For any object $M \in \Cal E_n$ such that all factor modules $\Ui$ are uniserial (possibly zero), consider the subset $A\subseteq \{1,\dots,n\}$ such that $\Ui =0$ if and only if $i \in A$. We want to define the following object $S(M)$ of $\Cal E_n$:
for every $i=1,\dots,n$, $S(M)^{(i)}$ is the direct sum of finitely many copies of $S$ in such a way that $S(M)^{(i)}/S(M)^{(i-1)}=0$ if $i \notin A$ and $S(M)^{(i)}/S(M)^{(i-1)}\cong S$ when $i \in A$. So, $M\oplus S(M)$ is an object of $\Cal U_n$. Notice that $S^n=S(0)$ and $S(M)=0$ if $M$ is an object of $\Cal U_n$. Using the canonical embeddings and the canonical projections, it is immediate to check that for every $a=m,e$ we have:
\begin{itemize}
	\item
	$[M]_{i,a}=[0]_{i,a}$ for all $i \in A$;
	\item
	the $i$-th factor of $M\oplus S(M)$ is isomorphic to
	$\Ui$ if $i \notin A$ and it is isomorphic to $S$ if
	$i \in A$;
	\item
	$[M\oplus S(M)]_{i,a}=[M]_{i,a}$ for all
	$i \notin A$;
	\item
	$[M\oplus S(M)]_{i,a}=[S(M)]_{i,a}=[S^n]_{i,a}$
	for all $i \in A$.
\end{itemize}

\begin{remark} \label{rema}
If $M$ is an object of $\Cal E_n$ with all the factor modules $\Ui$ uniserial, then it can be viewed as an object of $\Cal U_{n-d}$ for a suitable integer $d$ (namely, the number of indexes such that $\Ui=0$). It follows that $M$ is of finite type, because $E_M=\End_{\Cal E_n}(M)\cong \End_{\Cal E_{n-d}}(M)$ as rings. Moreover, the direct sum of finitely many object of $\Cal E_n$ with all factor modules uniserial has a semilocal endomorphism ring (Remark \ref{sumsemi}).
\end{remark} 

\begin{theorem}\label{main}
Let $M_1,\dots,M_r,N_1,\dots,N_s$ be non-zero objects of $\Cal E_n$ with uniserial factor modules $\Ui_k=M_k^{(i)}/M_k^{(i-1)}$ for $k=1,\dots,r$ and $\Vi_l=N_l^{(i)}/N_l^{(i-1)}$ for $l=1,\dots,s$. Define 
$$
X_i:=\{k\mid k=1,\dots\,r,\ \Ui_k\neq 0\}\ 
\mbox{ and } \ 
Y_i:=\{l\mid l=1,\dots\,s,\ \Vi_l \neq 0\}.
$$
Then $\bigoplus_{k=1}^r M_k \cong \bigoplus_{l=1}^s N_l$ in $\Cal E_n$ if and only if there exist $2n$ bijections $\varphi_{i,a}: X_i\rightarrow Y_i$ such that $[M_k]_{i,a}=[N_{\varphi_{i,a}(k)}]_{i,a}$ for every $i=1,\dots,n$, $a=m,e$.
\end{theorem}

\begin{proof}
Assume that $\bigoplus_{k=1}^r M_k \cong \bigoplus_{l=1}^s N_l$ in $\Cal E_n$. For any $i=1,\dots,n$, we have the following isomorphisms in Mod-$R$:
$$
\bigoplus_{k \in X_i}\Ui_k\cong \bigoplus_{k=1}^r \Ui_k\cong
\frac{(\bigoplus_{k=1}^r M_k)^{(i)}}{(\bigoplus_{k=1}^r M_k)^{(i-1)}}\cong
\frac{(\bigoplus_{l=1}^s N_l)^{(i)}}{(\bigoplus_{l=1}^s N_l)^{(i-1)}}
\cong \bigoplus_{l=1}^s \Vi_l \cong \bigoplus_{l \in Y_i}\Vi_l.
$$
So, looking at the Goldie dimension, we get that $|X_i|=|Y_i|$ for every $i=1,\dots,n$. Assume that $r\geq s$. Then, we have the following isomorphism in $\Cal E_n$
\begin{equation}\label{Siso}
\bigoplus_{k=1}^r S(M_k)\cong \left(
\bigoplus_{l=1}^s S(N_l) \right)  \oplus
\left( \underbrace{S^n \oplus S^n \oplus \dots \oplus S^n}_{(r-s)\text{-times}} \right)
\end{equation}
(notice that the $i^{th}$ submodule of both sides is isomorphic to the direct sum of $i r-(|X_1|+\dots+|X_i|)$ copies of $S$)
and so, by hypothesis, we have an isomorphism in $\Cal E_n$
$$
\bigoplus_{k=1}^r \left( M_k\oplus S(M_k)\right) \cong \left(
\bigoplus_{l=1}^s \left(N_l\oplus S(N_l) \right) \right)  \oplus
\left( \underbrace{S^n \oplus S^n \oplus \dots \oplus S^n}_{(r-s)\text{-times}} \right)
$$
where all the direct summands are in $\Cal U_n$. Write $M'_k=M_k\oplus S(M_k)$ for $k=1,\dots,r$, $N'_l=N_l\oplus S(N_l)$ for $l=1,\dots,s$ and $N'_l=S^n$ for $l=s+1,\dots,r$. From Proposition $\ref{parziale}$, there exist $2n$ permutations $\varphi_{i,a}$ of $\{1,2,\dots,r\}$, where $i=1,\dots,n$ and $a=m,e$, such that $[M'_k]_{i,a}=[N'_{\varphi_{i,a}(k)}]_{i,a}$. Let $k \in X_i$ and assume that $\varphi_{i,a}(k)\notin Y_i$. Since $|X_i|=|Y_i|$, there exists $j \notin X_i$ such that $\varphi_{i,a}(j)\in Y_i$. By construction, we have
$$
[M_k]_{i,a}=[M'_k]_{i,a}=[N'_{\varphi_{i,a}(k)}]_{i,a}=[S(N_{\varphi_{i,a}(k)})]_{i,a}=[S(M_j)]_{i,a}=[M'_j]_{i,a}=
$$
$$
[N'_{\varphi_{i,a}(j)}]_{i,a}=[N_{\varphi_{i,a}(j)}]_{i,a}
$$
and therefore, we can rearrange the permutation $\varphi_{i,a}$ in order to map $k$ into $\varphi_{i,a}(j)$ and $j$ into $\varphi_{i,a}(k)$, without changing its property of preserving the classes. In particular, we can always assume that $\varphi_{i,a}$ maps $X_i$ into $Y_i$, so that $\varphi_{i,a}\mid_{X_i}:X_i\rightarrow Y_i$ are the sought bijections. Similarly, for $r<s$.

Conversely, assume that there exist $2n$ bijections $\varphi_{i,a}: X_i\rightarrow Y_i$ such that $[M_k]_{i,a}=[N_{\varphi_{i,a})k=}]_{i,a}$ for $i=1,\dots,n$, $a=m,e$. Assume that $r\geq s$. Notice that the isomorphism in (\ref{Siso}) depends only on the fact that $|X_i|=|Y_i|$ for every $i =1,\dots,n$, so it holds also in this case. Consider the two direct sums
$$
\bigoplus_{k=1}^r \left( M_k\oplus S(M_k)\right)
$$
and
$$\left(
\bigoplus_{l=1}^s \left(N_l\oplus S(N_l) \right) \right)  \oplus
\left( \underbrace{S^n \oplus S^n \oplus \dots \oplus S^n}_{(r-s)\text{-times}} \right)
$$
and define, as before, $M'_k=M_k\oplus S(M_k)$ for $k=1,\dots,r$, $N'_l=N_l\oplus S(N_l)$ for $l=1,\dots,s$ and $N'_l=S^n$ for $l=s+1,\dots,r$. Fix $i=1,\dots,n$ and $a=m,e$. Then, for any $k \notin X_i$ and any $l \notin Y_i$, we have $[M'_k]_{i,a}=[S(M_k)]_{i,a}=[S(N_l)]_{i,a}=[N'_l]_{i,a}$, and therefore we can extend the bijection $\varphi_{i,a}$ to a bijection $\{1,2,\dots,r\}\rightarrow \{1,2,\dots,s,s+1,\dots,r\}$. Using Proposition \ref{parziale}, we get
$$
\bigoplus_{k=1}^r \left( M_k\oplus S(M_k)\right) \cong \left(
\bigoplus_{l=1}^s \left(N_l\oplus S(N_l) \right) \right)  \oplus
\left( \underbrace{S^n \oplus S^n \oplus \dots \oplus S^n}_{(r-s)\text{-times}} \right)
$$
in $\Cal E_n$. Now, applying \cite[Theorem 4.5]{libro}, we can cancel the two terms $\bigoplus_{k=1}^r S(M_k)$ and $\left(
\bigoplus_{l=1}^s S(N_l) \right)  \oplus
\left( S^n \oplus S^n \oplus \dots \oplus S^n \right)$  from direct sums, because they are isomorphic objects with semilocal endomorphism ring (Remark \ref{rema}). (Notice that Evans' result is stated in \cite[Theorem 4.5]{libro} only for modules, but its proof shows that it holds in any additive category.)
\end{proof}

\begin{example}
Notice that, in Theorem \ref{main}, dropping the assumption that all the objects are in  $\Cal U_n$, it could occur that $r\neq s$. A very simple example is the following. Let $M$ be any object of $\Cal E_n$ not in $\Cal U_n$ with all factor modules uniserial modules (some of whom are necessarily zero). Set $M_1=M\oplus S(M)$, $N_1=M$ and $N_2=S(M)$. Then, $M_1, N_1$ and $N_2$ are non-zero objects of $\Cal E_n$ and $M_1\cong N_1\oplus N_2$, but  $r=1$ and $s=2$. In particular, this example shows that the objects of $\Cal U_n$ are indecomposable objects in the category $\Cal U_n$, but they are not indecomposable in the category $\Cal E_n$.
\end{example}

\bibliographystyle{amsalpha}

\end{document}